\documentclass[12pt]{amsart}

\usepackage{amssymb,latexsym,amsmath,epsfig,amsthm} 
\usepackage{graphicx}
\usepackage[normalem]{ulem}
\usepackage{marvosym}
\usepackage{verbatim}
\newtheorem{definition}{Definition}
\newtheorem{theorem}{Theorem}
\newtheorem{lemma}{Lemma}

\newtheorem{corollary}{Corollary}
\newtheorem{example}{Example}
\newtheorem{remark}{Remark}

\begin{document}

\title{A note on graph compositions and their connection to minimax of set partitions}

\author{Todd Tichenor}

\maketitle

\begin{abstract}
A \emph{graph composition} is a partition of the vertex set such that each member of the partition induces a connected subgraph, and the \emph{composition number} of a graph is the number of possible graph compositions. A partition of a set $S$ of consecutive labelled vertices is said to have a \emph{minimax} $v\in S$ if the label of $v$ is the smallest label in the set of all maximum labels over all members of the partition. This paper exhibits a recursive formula for the composition number of a certain class of graphs and establishes a connection between the composition numbers of this class of graphs and that of the minimax of partitions of a labelled set (the minimum label of the set of all maximum labels over every member of the partition).
\end{abstract}

\section{Introduction} 
Graph compositions were introduced by Knopfmacher and Mays in \cite{KM}. Given a graph $G$, a \emph{graph composition} of $G$ is a partition of the vertex set such that each member induces a connected subgraph (called components). The composition number of a graph $G$ is the number of possible graph compositions. The composition number is denoted by $C(G)$. Graph compositions of $K_n$, a complete graph on $n$ vertices, can be thought of as merely an illustration of set partitions of the vertex set. Note that graph compositions and components of graph compositions are also graphs, and will be used in conjunction with other graph theory notation where necessary in this paper.
\\
\\
One area of research is the composition number of the graph $K_n^{-G}$, formed by removing the edges of $G$ from $K_n$ \cite{GCB}; specifically, previous work in \cite{TM} found formulae for certain classes of graphs which were deleted from $K_n$ (e.g. paths, cycles, stars, etc.). This paper focuses on the composition number of $K_n^{-K_m}$ for $0\leq m < n$ and its connection to the minimax of partitions of labelled vertices (formally defined later).  

\section{Composition number of $K_n^{-K_m}$}
\noindent We begin by establishing a recursive formula for $C(K_{n+m}^{-K_m})$.
\begin{theorem}
$C(K_{n+m}^{-K_m}) = \underset{i=0}{\overset{n-1}{\sum}}\binom{n-1}{i}\underset{j=0}{\overset{m}{\sum}}\binom{m}{j}C(K_{i+j}^{-K_j})$ for $n > m \geq 0$.
\end{theorem}

\begin{proof}
The graph $K_{n+m}^{-K_m}$ can be thought of as the join of $K_n$ and a graph $G_m$ of $m$ disjoint vertices (i.e. connecting every vertex of $G_m$ to all vertices of $K_n$). Mark a vertex $v\in V(K_n)$ and count all compositions of $K_{n+m}^{-K_m}$ by counting the number of components that contain $v$ and the number of compositions that contain the aforementioned component. Assuming that only $i$ vertices from $V(K_n)\backslash\{v\}$ and $j$ vertices from $M$ are missing from the component containing $v$, we have exactly $C(K_{i+j}^{-K_j})$ compositions that contain the component containing $v$. Summing over all possible choices for $i$ and $j$ yields $C(K_{n+m}^{-K_m}) = \underset{i=0}{\overset{n-1}{\sum}}\binom{n-1}{i}\underset{j=0}{\overset{m}{\sum}}\binom{m}{j}C(K_{i+j}^{-K_j})$.
\end{proof}

\begin{corollary}
 $C(K_n^{-K_m}) = \underset{i=0}{\overset{n-m-1}{\sum}}\binom{n-m-1}{i}\underset{j=0}{\overset{m}{\sum}}\binom{m}{j}C(K_{i+j}^{-K_j})$ for $n > m\geq 0$.
\end{corollary}

\begin{proof}
 Setting $n=n-m$ in the theorem above yields the result of the corollary.
\end{proof} 
~\\
A table of values for $C(K_n^{-K_m})$ follows.
\\
\begin{figure}[h!]
\centering
\begin{tabular}{|c||c|c|c|c|c|c|c}
\hline
$n\slash m$ & 0 & 1 & 2 & 3 & 4 & 5 & 6 \\ \hline\hline
0 & 1 & & & & & &\\ \hline
1 & 1 & 1 & & & & &\\ \hline
2 & 2 & 2 & 1 & & & &\\ \hline
3 & 5 & 5 & 4 & 1 & & &\\ \hline
4 & 15 & 15 & 13 & 8 & 1 & & \\ \hline
5 & 52 & 52 & 47 & 35 & 16 & 1 & \\ \hline
6 & 203 & 203 & 188 & 153 & 97 & 32 & 1 \\ 
\end{tabular}
\caption{Table of values of $C(K_n^{-K_m})$}
\label{fig1}
\end{figure}
\newpage
\begin{remark}
The partial table displayed seems to suggest that
\\ $\underset{m=0}{\overset{n}{\sum}}C(K_n^{-K_m}) = B_{n+1}$, the $(n+1)^{\text{st}}$ Bell number. Intuitively, this may seem strange since $C(K_n) = B_n$; however, this result is not in error, and will be explored further in the following section. 
\end{remark}
\section{Connection to the minimax of set partitions}
\begin{definition}
Given a set $S=\{1,2,...,n\}$, $m\in S$ is said to be the \emph{minimax} of a partition of $S$ if its label is the smallest  in the set of all maximum labels over all members of the partition. The number of partitions of $S$ such that $m$ is the minimax of the partition is denoted by $T(n,m)$ \cite{OEIS}.
\end{definition} 
An example of the minimax of partitions follows.
\begin{example}
$S  = \{1,2,3\}$
\begin{figure}[h!]
\centering
\begin{tabular}{|c|c|c|}
\hline
Partition of $S$ & Set of maximum labels & Minimax \\ \hline
$\{\{1,2,3\}\}$ &$\{3\}$ & 3\\ \hline
$\{\{1\},\{2,3\}\}$ &$\{1,3\}$ & 1\\ \hline
$\{\{2\},\{1,3\}\}$ &$\{2,3\}$ & 2\\ \hline
$\{\{3\},\{1,2\}\}$ &$\{2,3\}$ & 2\\ \hline
$\{\{1\},\{2\},\{3\}\}$ &$\{1,2,3\}$ &1\\ \hline
\end{tabular}
\caption{Example of the minimax}
\end{figure}
\end{example}
~\\
~\\
It is known that $T(n,m) = \underset{k=1}{\overset{m}{\sum}}S(m-1,k-1)\cdot k^{n-m}$, where $S(m,k)$ represents the $(m,k)^{\text{th}}$ entry in the array of Stirling numbers of the second kind  \cite{OEIS}. 
\\
An equivalent definition of the minimax in the context of graph compositions is defined as follows.
\begin{definition}
If $G$ is a labelled graph with $V(G) = \{1,2,...,n\}$, then $v\in V(G)$ is the minimax vertex of a graph composition of $G$ if its label is the smallest in the set of all maximum labels over all components of the composition.
\end{definition}
~\\
We will use $k(n,m)$ to denote the number of compositions of $K_n$ which has minimax vertex $m$. It is trivial to note that $k(n,m) = T(n,m)$; note also that the minimax vertex of a graph composition is unique which allows us to sort graph compositions via the minimax vertex. This idea, although trivial, is important later in this work. 

\begin{theorem} \label{thm1}
$C(K_n^{-K_m})=k(n+1,m+1) $.
\end{theorem}

\begin{proof}The result will be proved by establishing a bijection between the set of all compositions of $K_{n+1}$ with minimax vertex $m+1$ and some graph isomorphic to $K_n^{-K_m}$. Let $M$ denote the set of vertices $\{1,2,...,m\}$ and $G$ be the graph such that $V(G)= V(K_{n+1})\backslash\{m+1\}$
and $E(G) = \{\{uv\}: u\not\in M \text{ or } v\not\in M\}$. Observe that $G \cong K_n^{-K_m}$. Consider a composition $C$ of $K_{n+1}$ such that the minimax vertex is $m+1$ and denote the component containing $m+1$ by $\mathcal{C}_{m+1}$. It is necessary that no component of $C$ consists solely of vertices from $M$ and that $V(\mathcal{C}_{m+1})\backslash M = \{m+1\}$, otherwise $m+1$ will not be the minimax vertex. The deletion of $m+1$ and all its incident edges from $C$ will yield a unique composition of $G$; conversely, adding $m+1$ to any composition of $G$ and connecting it to all singleton vertices of $M$ will yield a unique composition of $K_{n+1}$ such that the minimax vertex is necessarily $m+1$. Hence the theorem is proved.
\end{proof}

\begin{remark}
Theorem \ref{thm1} also provides an explicit formula for $C(K_n^{-K_m})$; namely $C(K_n^{-K_m}) = k(n+1,m+1) = T(n+1,m+1) = \underset{k=1}{\overset{m+1}{\sum}}S(m,k-1)\cdot k^{n-m}$.
\end{remark}
\begin{corollary}
$\underset{m=0}{\overset{n}{\sum}}C(K_n^{-K_m})  = B(n+1)$. 
\end{corollary}

\begin{proof}
It is obvious that since compositions have a unique minimax value, all compositions of $K_{n+1}$ can be sorted by the aforementioned value. Summing  all compositions of $K_{n+1}$ over all possible minimax values will yield the total number of compositions; hence, with the help of Theorem \ref{thm1},\\ $B(n+1) = C(K_{n+1}) = \underset{m=1}{\overset{n+1}{\sum}}k(n+1, m) = \underset{m=0}{\overset{n}{\sum}}k(n+1,m+1) = \underset{m=0}{\overset{n}{\sum}}C(K_n^{-K_m})$.
\end{proof}
~\\
One final avenue which will be explored here is a possible extension of the concept of the minimax vertex. In the definition given, minimax vertices were defined without regard to the number of vertices in a component. This concept is now considered with a restriction placed on the number of vertices in the component containing the minimax vertex. Let $k_j(n,m)$ denote the number of compositions of $K_n$ such that $m$ is the minimax vertex of all components containing no more than $j$ vertices. It was shown in \cite{TM} that $C(K_n^{-G})$ can be expressed as the sum of multiples of the Bell numbers. The full result is given as the following lemma. 
\begin{lemma}
Let $G$ be a subgraph of $K_N$ such that $|V(G)| = n$ and $b_{j,k,n}$ denote the number of ways of choosing $k$ disjoint components of $K_N^{-G}$ such that the cardinality of the union of vertices of all $k$ components is $j$. If we define $b_{j,n} = \underset{k=0}{\overset{n}{\sum}}(-1)^k\cdot b_{j,k,n}$, then $C(K_N^{-G}) = \underset{j=0}{\overset{n}{\sum}}b_{j,n}B_{n-j}$.
\end{lemma}

Given this lemma, it should not be surprising that the composition number of any graph is the sum of multiples of the Bell numbers (since any graph can be viewed as the deletion of its complement from some $K_n$). Taking the lemma above and restricting the type of compositions counted by $b_{j,k,n}$, it should not be surprising that $k_j(n,m)$ in general can also be expressed as the sum of multiples of the Bell numbers. The details of this do not yield any insight and have been omitted; however, the formula for $k_1(n,m)$ is proved below so that formula for the extreme values of $j$ are provided.

\begin{theorem}
 $k_1(n,m) = \underset{j=1}{\overset{m}{\sum}}(-1)^{j+1}\binom{m-1}{j-1}B(n-j)$.
\end{theorem}

\begin{proof}
 If $C_m$ denotes the set of all compositions of $K_n$ which contains the singleton component $\{m\}$, then $C_m\backslash\underset{j=1}{\overset{m-1}{\cup}}C_j$ is the set of compositions of $K_n$ such that $\{m\}$ is the minimum singleton component. This statement implies that $k_1(n,m) = |C_m\backslash\underset{j=1}{\overset{m-1}{\cup}}C_j| = |C_m| - |\underset{j=1}{\overset{m-1}{\cup}}C_j|$.
By the inclusion-exclusion principle,
\\
\\
$$|\underset{j=1}{\overset{m-1}{\cup}}C_j| = \underset{\emptyset \not= J\subseteq \{1,2,...,m-1\}}{\sum}(-1)^{|J|+1}|\underset{j\in J}{\cap}C_j|.$$
The set $\underset{j\in J}{\cap}C_j$ is the set of all compositions of $K_n$ which contain the elements of $J$ and $\{m\}$ as singleton components, which yields \\
$|\underset{j\in J}{\cap}C_j| = B(n - (|J| + 1))$.  Adding that the number of ways of choosing $J$ of fixed cardinality from $\{1,2,...,m-1\}$ is $\binom{m-1}{|J|}$, we get the result
$ |\underset{j=1}{\overset{m-1}{\cup}}C_j| = \underset{j=1}{\overset{m-1}{\sum}}(-1)^{j+1}\binom{m-1}{j}B(n-(j+1))$. Hence
\\
\begin{align*}
 k_1(n,m) &= B(n-1) + \underset{j=1}{\overset{m-1}{\sum}}(-1)^j\binom{m-1}{j}B(n-(j+1)) \\
&= B(n-1) +  \underset{j=2}{\overset{m}{\sum}}(-1)^{j+1}\binom{m-1}{j-1}B(n-j) \\
&=  \underset{j=1}{\overset{m}{\sum}}(-1)^{j+1}\binom{m-1}{j-1}B(n-j)
\end{align*}
\end{proof} 
\noindent A table of values for $k_1(n,m)$ concludes the paper.
\\
\begin{figure}[!ht]
\centering
\begin{tabular}{|c||c|c|c|c|c|c|c|c|c}
\hline
$n\slash m$ & 0 & 1 & 2 & 3 & 4 & 5 & 6 & 7 & 8 \\ \hline\hline
1 & 0 & 1 & & & & & & &\\ \hline
2 & 1 & 1 & 0 & & & & & &\\ \hline
3 & 1 & 2 & 1 & 1 & & & & &\\ \hline
4 & 4 & 5 & 3 & 2 & 1 & & & & \\ \hline
5 & 11 & 15 & 10 & 7 & 5 & 4 & & & \\ \hline
6 & 41 & 52 & 37 & 27 & 20 & 15 & 11 & & \\ \hline
7 & 162 & 203 & 151 & 114 & 87 & 67 & 52 & 41 &\\ \hline
8 & 715 & 877 & 674 & 523 & 409 & 322 & 255 & 203 & 162\\
\end{tabular}
\caption{Table of values for $k_1(n,m)$}
\end{figure}

\end{document}